\documentclass[oneside,a4paper,11pt,notitlepage]{article}
\usepackage[left=0.8in, right=0.8in, top=1.5in, bottom=1.5in]{geometry}
\usepackage{pifont}
\usepackage{makecell}
\usepackage[T1]{fontenc} 
\usepackage[utf8]{inputenc} 
\usepackage[english]{babel}
\usepackage{lipsum} 
\usepackage{lmodern}
\usepackage{amssymb}
\usepackage{amsthm}
\usepackage{bm}
\usepackage{mathtools}
\usepackage{braket}
\usepackage{esint}
\newcommand{\abs}[1]{{\left|#1\right|}}

\usepackage{tabularx}
\usepackage{fontawesome}
\usepackage{booktabs}
\usepackage{graphicx}
\usepackage{tikz}
\usetikzlibrary{patterns}
\usepackage{multicol}
\usepackage{caption}
\usepackage{enumerate}
\usepackage{multicol}
\usepackage[skins,theorems]{tcolorbox}
\tcbset{highlight math style={enhanced,
		colframe=black,colback=white,arc=0pt,boxrule=1pt}}
\def\XXint#1#2#3{{\setbox0=\hbox{$#1{#2#3}{\int}$}
    \vcenter{\hbox{$#2#3$}}\kern-.5\wd0}}

\theoremstyle{definition}
\newtheorem{definizione}{Definition}[section]
\theoremstyle{plain}
\newtheorem{teorema}{Theorem}[section]

\newtheorem{prop}[teorema]{Proposition}
\newtheorem{corollario}[teorema]{Corollary}
\theoremstyle{definition}
\newtheorem{esempio}{Example}[section]
\newtheorem{oss}[esempio]{Remark}
\newtheorem{congettura}[esempio]{Conjecture}

\DeclareMathOperator{\R}{\mathbb{R}}

\DeclareMathOperator{\diam}{\, \textup{diam}}

\makeatletter
\newcommand{\myfootnote}[2]{\begingroup
	\def\@makefnmark{}%
	\addtocounter{footnote}{-1}%
	\footnote{\textbf{#1} #2}
	\endgroup}
\makeatother

\usepackage{soul}

\definecolor{OliveGreen}{rgb}{0,0.6,0}

\usepackage{hyperref}
\hypersetup{linktoc=none, bookmarksnumbered, colorlinks=true, linkcolor=red}

 \title{The Makai inequality in higher dimensions: qualitative and quantitative aspects}
\author{V. Amato, N. Gavitone, R. Sannipoli}
\date{}

\newcommand{\Addresses}{{
  \bigskip 
   \footnotesize 
 \noindent \textit{E-mail address}, V.~ Amato: \texttt{v.amato@ssmeridionale.it} 
  
   \medskip 
 
  \noindent\textsc{Mathematical and Physical Sciences for Advanced Materials and Technologies, Scuola Superiore Meridionale, Largo San Marcellino 10, 80138 Napoli, Italy. }

 \medskip

 \noindent  \textit{E-mail address}, N.~Gavitone: \texttt{nunzia.gavitone@unina.it} 
   
 \medskip
\noindent  \textsc{Dipartimento di Matematica e Applicazioni ``R. Caccioppoli'', Universit\`a degli studi di Napoli Federico II, Via Cintia, Complesso Universitario Monte S. Angelo, 80126 Napoli, Italy.}
  \medskip
 
 \noindent \textit{E-mail address}, R.~Sannipoli: \texttt{rossano.sannipoli@fjfi.cvut.cz} 
  
     \medskip 
\noindent\textsc{Department of Mathematics, Faculty of Nuclear Sciences and Physical Engineering, Czech Technical University in Prague, Trojanova 13, 120 00, Prague, Czech Republic.}

\par\nopagebreak 

}} 
\makeatletter
\def\Cline#1#2{\@Cline#1#2\@nil}
\def\@Cline#1-#2#3\@nil{%
  \omit
  \@multicnt#1%
  \advance\@multispan\m@ne
  \ifnum\@multicnt=\@ne\@firstofone{&\omit}\fi
  \@multicnt#2%
  \advance\@multicnt-#1%
  \advance\@multispan\@ne
  \leaders\hrule\@height#3\hfill
  \cr}
\makeatother

\definecolor{verde}{RGB}{20,150,100}
\definecolor{purple}{RGB}{200,30,200}

\begin{document}

\maketitle
\begin{abstract} 
  In this paper, given a convex, bounded, open set $\Omega \subset \R^n$ we prove a sharp inequality involving the Laplacian torsional rigidity  and both the perimeter and the measure  of the domain. 
  Our result generalizes to arbitrary dimensions the inequality established by Makai in the plane which, as conjectured in \cite{buttazzo2020convex}. Furthermore, we establish quantitative  estimates that provide key insights into the geometric structure and the thickness of the underlying optimizing sequences.

\end{abstract}

\section{Introduction}
Let $n \geq 2$. We denote by $\mathcal{K}_n$ the class of all non-empty, open, bounded, and convex subsets of $\mathbb{R}^n$. Given $\Omega \in \mathcal{K}_n$, in this paper we study the maximization problem 
\begin{equation*}
\sup\{{\mathcal{F}(\Omega),\;\; \Omega \in \mathcal{K}_n}\},
\end{equation*}
where $\mathcal{F}:\mathcal{K}_n\to\mathbb R^+$ is the following scaling invariant functional
\begin{equation}
\label{mak_func}
\mathcal F(\Omega)=\frac{T(\Omega) P^2(\Omega)}{|\Omega|^3}.
\end{equation}
Here $P(\Omega)$ and $|\Omega|$ denote the perimeter and the Lebesgue measure of $\Omega$, respectively, and $T(\Omega)$ is the torsional rigidity of the Laplacian in $\Omega$, defined by
$$
T(\Omega) =  \max_{\substack{\varphi \in H^1_0(\Omega) \\ \varphi \not\equiv 0}} 
\frac{\displaystyle\left(\int_\Omega \varphi \, dx \right)^2}{\displaystyle \int_\Omega |\nabla \varphi|^2 \, dx}.
$$
The scaling invariance of $\mathcal F$ comes from these properties valid for every $t>0$
\[
T(t\Omega)= t^{n+2}T(\Omega), \qquad P(t\Omega) =t^{n-1}P(\Omega), \qquad |t\Omega|=t^n |\Omega|.
\]

In general, torsional rigidity is not explicitly computable, except for highly symmetric domains; this motivates the derivation of estimates in terms of other geometric quantities in order to better understand its behavior. Pioneering contributions in this direction are due to Pólya and Makai, who introduced different scaling invariant functionals capturing the interplay between torsion and the geometry of the domain, such as $\mathcal F(\Omega)$.

As regards the minimization problem for $\mathcal F$, Pólya proved in \cite{polya1960}, in the planar case, the following lower bound:
\begin{equation}
\label{polyatorsion}
\frac{1}{3} \leq \frac{T(\Omega) P^2(\Omega)}{|\Omega|^3}.
\end{equation}
Moreover, he showed that the bound is asymptotically attained by sequences flattening rectangles. This inequality was later extended to higher dimensions in \cite{gavitone_2014}, where it was proved for every $\Omega \in \mathcal{K}_n$ and shown to be asymptotically attained by sequences of flattening cylinders, i.e., cylinders whose height tends to zero.\\ 
We now turn to the maximization problem. Makai proved in \cite{makai}, in the planar case, the following sharp upper bound for $\mathcal F$ within the class $\mathcal{K}_2$:
\begin{equation}
\label{makaitorsion}
\frac{T(\Omega) P^2(\Omega)}{|\Omega|^3} \leq \frac{2}{3}.
\end{equation}
Moreover, he showed that the inequality is optimal, in the sense that equality is asymptotically attained along sequences of flattening triangles. In dimension two, several extensions of this result have been obtained, also in nonlinear settings (see for instance \cite{dellapietra_gavitone2018,fragala_gazzola_lamboley2013,HLP2018}). In \cite{buttazzo2020convex}, the authors formulate the following conjecture for the sharp upper bound in arbitrary dimension:
\begin{congettura}\label{congettura}
\[
\sup\{\mathcal F(\Omega) : \Omega \in \mathcal{K}_n\}
= \frac{2n^2}{(n+1)(n+2)}.
\]
\end{congettura}
They also provide partial evidence supporting the conjecture by analyzing suitable families of thin domains. In particular, they are able to show that sequences of flattening cones attain this exact value.
\\
We are now in a position to state the first main result of this paper, which confirms the above conjecture.

\begin{teorema}\label{thm:main}
Let $n \geq 2$ and let $\Omega\in \mathcal{K}_n$. Then
\begin{equation*}
\mathcal F(\Omega) \leq \frac{2n^2}{(n+1)(n+2)}.
\end{equation*}
Moreover, the constant is sharp and is asymptotically achieved, for instance, by a sequence of flattening cones.
\end{teorema}

Since Briani, Buttazzo, and Prinari \cite{buttazzo2020convex} show that the constant is attained by sequences of flattening cones, once we prove the inequality it is immediately sharp.

As stated in Theorem \ref{thm:main}, the optimal constant is achieved by a particular sequence of sets in $\mathcal{K}_n$. A natural question that could arise is whether the supremum of $\mathcal{F}$ is actually a maximum. The answer to this question is negative and it is explained in the next result. To state it, we need to introduce a remainder term already introduced in different papers (see for instance \cite{ABF,ABF2, AGS,AMPS}), that is the following
\begin{equation}\label{alpha}
 \alpha(\Omega) := \frac{w_\Omega}{\diam(\Omega)},
\end{equation}
where $w_\Omega$ and $\diam(\Omega)$ are respectively the minimal width and the diameter of $\Omega$ (see Section \ref{sec2} for the precise definition). We emphasize that the sets for which $\alpha(\Omega)\to 0$ identifies the class of so-called \textit{thinning domains} (see Section \ref{sec2}). Our second main result in this direction is the following.
\begin{teorema}\label{thm:supremum}
 Let $n\ge 2$. If $\{\Omega_k\}_{k\in \mathbb N}\subset K_n$  is a maximizing sequence for $\mathcal{F}(\Omega)$, i.e. such that 
    \begin{equation*}\label{eq:SharpMakai}
        \lim_{k\to +\infty}\mathcal{F}(\Omega_k)= \frac{2n^2}{(n+1)(n+2)},
    \end{equation*}
then we get
    $$\lim_{k\to+\infty} \alpha(\Omega_k)=0.$$
    \end{teorema}
Clearly Theorem \ref{thm:supremum} implies that the supremum of $\mathcal{F}(\Omega)$ is not attained by any set. Indeed if there existed a $\tilde{\Omega}\in \mathcal{K}_n$, such that $\mathcal{F}(\tilde{\Omega})= \frac{2n^2}{(n+1)(n+2)}$, then we could choose as a maximizing sequence the constant one $\Omega_k = \tilde\Omega$, for all $k\in \mathbb N$. Applying Theorem \ref{thm:supremum}, we would have $\alpha(\tilde\Omega)=0$ and therefore that the interior of $\tilde\Omega$ is empty, but this is not allowed in $\mathcal{K}_n$.\\

Actually, as proved in \cite{AGS}, not every thinning domain behaves in the same way and then it is possible to introduce different thickness parameters in order to give a finer description of minimizing sequences. More precisely, in \cite[subsection $2.4$]{AGS}, we introduced the thickness parameters
\begin{equation}
\label{beta}
\beta(\Omega):=\frac{P(\Omega)R_\Omega}{|\Omega|}-1 \in(0,n-1].
\end{equation}  
In particular, in \cite[Proposition $1.1$]{AGS}, we proved that there exists a positive dimensional constant $C_0=C_0(n)$, such that
\begin{equation*}
    \beta(\Omega)\ge C_0\alpha(\Omega),
\end{equation*}
while the converse inequality cannot be true, since there are sequences of thinning domains for which the functional $\beta(\Omega)$
is not converging to zero (for instance a sequence of flattening cones).
What we proved in \cite{AGS} is the following inequalities for the minimization problem of $\mathcal{F}(\Omega)$
\begin{equation}\label{eq:TPM-PRM}
\frac{n+1}{3}\beta(\Omega)\ge 
\mathcal{F}(\Omega)-\frac{1}{3} \ge \frac{1}{2^3\cdot 3^4 n^3} \beta(\Omega)^3.
\end{equation}    
This result, which is a continuity result and a quantitative versione of the P\'olya inequality, provides a precise characterization of the thinning minimizing sequences for inequality \eqref{polyatorsion} in terms of the purely geometric quantity $\beta(\Omega)$, providing a more detailed description of the thickness.\\

Motivated by the quantitative approach described above, the second aim of this paper is to characterize the maximing sequence of $\mathcal{F}(\Omega)$, investigating whether the deficit from the optimal value can be related to suitable geometric quantities. Theorem \ref{thm:supremum} guarantees that the maximing sequence must be thinning and then, as in the spirit of \eqref{eq:TPM-PRM}, we introduce a new  remainder term, in order to give a more detailed characterization of the optimizing sequence, that is the following
\begin{equation}\label{gamma}
    \gamma(\Omega) = n-\frac{P(\Omega)R_\Omega}{\abs{\Omega}}\in [0,n-1).
\end{equation}
We stress that $\gamma(\Omega)= 0$ on tangential bodies and that $\gamma(\Omega)\to n-1$ when $\beta(\Omega) \to 0$, i.e. on sequences of thinning cylinders (see Section \ref{sec2} for the details).\\ 
In this spirit, we establish a quantitative version of Theorem \ref{thm:main}, that is the following.

\begin{teorema}\label{thm:quantitative}
    Let $n\ge 2$ and let $\Omega \in \mathcal{K}_n$. Then,
    \begin{equation*}
     \frac{2n^2}{(n+1)(n+2)}-\mathcal{F}(\Omega)\geq C_1\gamma(\Omega)^n
    \end{equation*}
      where
    \begin{equation*}
        C_1(n) = \frac{n^2+3n-4}{n(n+1)(n+2)(n-1)^{n}}.
    \end{equation*}
\end{teorema}
We point out that the quantitative estimate in Theorem \ref{thm:quantitative} does not involve the parameter $\alpha(\Omega)$. But Theorem \ref{thm:supremum} together with Theorem \ref{thm:quantitative} allow to characterize all the maximizing sequence for $\mathcal{F}(\Omega)$, as sequences for which both $\alpha(\Omega)$ and $\gamma(\Omega)$ go to zero, that is
\begin{corollario}
\label{Corol}
A maximizing sequence in $\mathcal{K}_n$ for $\mathcal{F}(\Omega)$  is a sequence of flattening tangential bodies.
\end{corollario}

Simultaneously and independently of our work, Pisante and Prinari  \cite{PP} established sharp inequalities for the Poincaré--Sobolev constants $\lambda_{p,q}(\Omega)$ in convex domains, which are closely related to the qualitative result of our paper.

\vspace{0.5cm}
\textbf{Plan of the paper:} In Section \ref{sec2}, we recall some basic notions and definitions and review some classical results, focusing in particular on the class of convex sets. In Section \ref{sec3}, we prove Theorem \ref{thm:main}. In Section \ref{sec4}, we further investigate the geometry of the optimizing sequences, proving Theorems \ref{thm:quantitative} and \ref{thm:supremum}.

\section{Notations and preliminary results}
\label{sec2}
\subsection{Notations and basic facts} 
 Throughout this article, $|\cdot|$ will denote the Euclidean norm in $\mathbb{R}^n$,
 while $(\,\cdot\,)$ is the standard Euclidean scalar product for  $n\geq2$. For $k\in [0,n)$, the $k-$dimensional Hausdorff measure is denoted by $\mathcal{H}^k(\cdot)$. We will denote by $B_r(x)$ the ball centered at the point $x\in \mathbb R^n$ with radius $r>0$; moreover, when the ball is centered at the origin we will write $B_r$, omitting the dependence on $x$, and when the radius of the ball is $1$, we will denote by $\mathbb{S}^{n-1}$ its boundary.  We denote by $\mathcal{K}_n$ the class of all non-empty, open, bounded, and convex subsets of $\mathbb{R}^n$. Let $\Omega \in \mathcal{K}_n$, its perimeter can be defined by
\[
P(\Omega) = \mathcal{H}^{n-1}(\partial \Omega).
\]

 We give now the definition of the support function of a convex set and minimal width (or thickness) of a convex set (see for instance \cite{schneider}).
\begin{definizione}\label{support}
  Let $\Omega\in \mathcal{K}_n$. The support function of $\Omega$ is defined as
  \begin{equation*}
    h_\Omega(y)=\sup_{x\in \Omega}\left(x\cdot y\right), \qquad y\in \mathbb{R}^n .
  \end{equation*}
\end{definizione}

\begin{definizione}
Let $\Omega\in \mathcal{K}_n$, the width of $\Omega$ in the direction $y \in \mathbb{S}^{n-1}$ is defined as 
  \begin{equation*}
    \omega_{\Omega}(y)=h_{\Omega}(y)+h_{\Omega}(-y)
    \end{equation*}
 and the minimal width of $\Omega$ as
\begin{equation*}
    w_\Omega=\min\{  \omega_{\Omega}(y)\,|\; y\in\mathbb{S}^{n-1}\}.
\end{equation*}
\end{definizione}
We will denote by $R_\Omega$ is the inradius of $\Omega$, i.e.
 \begin{equation}
     \label{inradiuss}
     R_\Omega=\sup\{r\in \mathbb R: B_r(x)\subset \Omega, x\in \Omega\},
 \end{equation}
and by $\diam(\Omega)$ the diameter of $\Omega$, that is
\begin{equation*}
    \diam(\Omega) = \sup_{x,y\in\Omega}|x-y|.
\end{equation*}

 \subsection{Distance function and inner parallel sets} \label{innere}
Let $\Omega\in \mathcal{K}_n$. We define the distance function from the boundary, and we will denote it by $ d(\cdot, \partial \Omega):\Omega \to [0,+\infty[$, as follows 
 \begin{equation*}
     d(x,\partial\Omega):=\inf_{y\in\partial\Omega}\abs{x-y},
 \end{equation*}
 and we call inradius, $R_\Omega$, its maximum.
 
We remark that the distance function is concave, as a consequence of the convexity of $\Omega$.
The superlevel sets of the distance function
\begin{equation*}
    \Omega_t=\set{x\in \Omega \, : \, d(x,\partial\Omega)>t}, \qquad t\in[0, R_\Omega]
\end{equation*}
 are called \emph{inner parallel sets}, and we use the following notations:
 \begin{equation*}
 \mu(t)= \abs{\Omega_t}, \qquad P(t)=P(\Omega_t)\qquad t\in[0, R_\Omega].
 \end{equation*}
 By coarea formula (see \cite{maggi2012sets}), recalling that $\abs{\nabla d}=1$, it is possible to prove that $\mu(t)$ is absolutely continuous, decreasing and its derivative is  
\begin{equation}\label{eq:dermu}
    \mu'(t)=-P(t)\qquad a.e.
\end{equation}
By the Brunn-Minkowski inequality (\cite[Theorem 7.4.5]{schneider}) and the concavity of the distance function, the map
\begin{equation*}
    t \mapsto P(t)^{\frac{1}{n-1}}
\end{equation*}
is concave in $[0,R_{\Omega}]$, hence absolutely continuous in $(0,R_{\Omega})$. Moreover, there exists its right derivative at $0$ and it is negative, since $P(t)^{\frac{1}{n-1}}$ is strictly monotone decreasing, hence almost everywhere differentiable. For the next definition see \cite{schneider}.
\begin{definizione}
Let $\Omega \in \mathcal{K}_n$. We will say that $\Omega$ is a tangential body if its inner parallel sets are homothetic to $\Omega$ itself, i.e. if $\Omega_{t}= (1-\frac{t}{R_\Omega}) \Omega $ for $t \in (0,R_\Omega)$.
\end{definizione}
The following proposition concerns inequalities that are implied by the concavity of the functions $\abs{\Omega}^{\frac1n}$ and $P(\Omega)^{\frac{1}{n-1}}$, which are concave by the Brunn-Minkowski inequality. 

\begin{prop}
\label{prop:concavitymeasure}
For every $t \in [0,R_\Omega]$, we have
\begin{equation}\label{eq:concavity}
\begin{split}
& \mu(t)\ge \left(1-\frac{t}{R_\Omega}\right)^{n}\abs{\Omega}\\
&P(t) \ge \left(1-\frac{t}{R_\Omega}\right)^{n-1} P(\Omega).
  \end{split}  
\end{equation}
Moreover the equality case is achieved by tangential bodies.

\end{prop}

\begin{proof}
Since the functions $\mu(t)^{\frac{1}{n}}$  $P(t)^{1/(n-1)}$ are concave in $[0,R_\Omega]$, then they are above the lines passing through the points $(0,P(\Omega)^{\frac{1}{n-1}})$ and $(R_\Omega, P(R_\Omega)^{\frac{1}{n-1}})$, that have respectively equations
\begin{equation*}
    y_m(t) = -\frac{\abs{\Omega}^{\frac{1}{n}}}{R_\Omega}t+\abs{\Omega}^{\frac{1}{n}}= \bigg(1-\frac{t}{R_\Omega}\bigg)\abs{\Omega}^{\frac{1}{n}},
\end{equation*}
and
\begin{equation*}
    y_p(t) = \frac{P(R_\Omega)^{\frac{1}{n-1}}-P(\Omega)^{\frac{1}{n-1}}}{R_\Omega}t+P(\Omega)^{\frac{1}{n-1}}\ge \bigg(1-\frac{t}{R_\Omega}\bigg)P(\Omega)^{\frac{1}{n-1}}.
\end{equation*}
Therefore we have
\begin{align*}
&\mu(t)^{\frac1n}\ge \bigg(1-\frac{t}{R_\Omega}\bigg)\abs{\Omega}^{\frac{1}{n}}\\
&P(t)^{\frac{1}{n-1}}
\ge
\left(1-\frac{t}{R_\Omega}\right) P(\Omega)^{\frac{1}{n-1}}.
\end{align*}
Raising both sides to the power $n$ and $n-1$ respectively yields \eqref{eq:concavity}.
\end{proof}
The following result gives an upper bound on the torsion in terms of the distance function. It was proved, for instance, in \cite{Cra2004}, but for the reader’s convenience we also include its proof. See also \cite{prinari2023sharp} for a more general inequality.
\begin{prop}
\label{prop:TlessD}
    Let $\Omega$ be a bounded, open and convex set of $\mathbb{R}^n$. Then
    \begin{equation}\label{eq:torsiondistancefunction}
        T(\Omega)\le \int_{\Omega}d(x,\partial\Omega)^2\,dx.
    \end{equation}
\end{prop}

\begin{proof}
   We will give the proof in the case of $\Omega$ being a convex polytope, which is defined as the convex hull of finitely many points in $\mathbb R^n$. The proof for a general open, bounded convex set will follow by approximation arguments (see \cite[Theorem $1.8.16$]{schneider}). Let us assume that the polytope $\Omega$ has $N\in\mathbb N$ facets, denoted by $F_i$, $i=1,..., N$, and let us decompose
   \begin{equation*}
      \Omega= \bigcup_{i=1}^N E_i,
   \end{equation*}
   where $E_i$ is a connected component of $\Omega$, where the map $x\to d(x,\partial\Omega)$ is differentiable.\\
   If $v$ is the torsion function in $\Omega$, then denoting by $\nabla_{x'}v= (\partial_1v,...,\partial_{n-1}v)$, where  $x'=(x_1,...,x_{n-1})$ and by $\partial_kv= \frac{\partial v}{\partial x_k}$, $k=1,...,n$, then by Cauchy's inequality we obtain
   \begin{equation*}
       T(\Omega) =\frac{\left(\displaystyle{\sum_{i=1}^N\int_{E_i} v \, dx }\right)^2}{\displaystyle{\sum_{i=1}^N\int_{E_i}(\abs{\nabla_{x'} v}^2+(\partial_nv)^2) \, dx }}\le \sum_{i=1}^N\frac{\displaystyle{\left(\int_{E_i} v \, dx \right)^2}}{\displaystyle{\int_{E_i}(\abs{\nabla_{x'} v}^2+(\partial_nv)^2) \, dx }}.
   \end{equation*}
   Without loss of generality, we can assume that $F_i$ lies on the $\mathbb R^{n-1}$ hyperplane, so that $E_i$ can be parametrized as follows
   \begin{equation*}
       E_i = \{(x',y): x'\in F_i, 0<y<f_i(x')\},\qquad i=1,...,N,
   \end{equation*}
   where $f_i:x'\in F_i\to \mathbb R$ parametrizes $\partial E_i \cap \Omega$. Clearly $F_i = \overline{E_i} \cap \{y=0\}$. In this way
   \begin{equation*}
       T(E_i) := \frac{\displaystyle{\left(\int_{E_i} v \, dx \right)^2}}{\displaystyle{\int_{E_i}(\abs{\nabla_{x'} v}^2+(\partial_nv)^2) \, dx }}\le \frac{\displaystyle{\left(\int_{E_i} v \, dx \right)^2}}{\displaystyle{\int_{E_i}(\partial_nv)^2 \, dx }}=\frac{\displaystyle{\left(\int_{F_i}dx'\int_0^{f_i(x')} v \, dx_n \right)^2}}{\displaystyle{\int_{F_i}dx'\int_0^{f_i(x')}(\partial_nv)^2 \, dx_n }}. 
   \end{equation*}
   Now, using Cauchy-Schwarz inequality and the following inequality (see \cite[equation $(7)$]{makai})
   \begin{equation*}
       \frac{\displaystyle{\bigg(\int_0^sg(t)\,dt\bigg)^2}}{\displaystyle{\int_0^s(g'(t))^2\,dt}}\le \frac{s^3}{3}
   \end{equation*}
   valid for any function $g\in C^1([0,s])$, such that $g(0)=0$, we get
   \begin{equation*}
       T(E_i) \le \int_{F_i}\frac{\displaystyle{\left(\int_0^{f_i(x')} v \, dx_n \right)^2}}{\displaystyle{\int_0^{f_i(x')}(\partial_nv)^2 \, dx_n }}dx'\le \int_{F_i}\frac{f_i(x')^3}{3}\,dx'= \int_{F_i}\int_0^{f_i(x')}y^2\,dydx'= \int_{E_i}d(x,\partial \Omega)^2\,dx .
   \end{equation*}
   Eventually
   \begin{equation*}
       T(\Omega) \le \sum_{i=1}^NT(E_i)\le\sum_{i=1}^N\int_{E_i}d(x,\partial \Omega)^2\,dx  = \int_{\Omega}d(x,\partial \Omega)^2\,dx .
   \end{equation*}
   
\end{proof}

\begin{oss}
\label{crasta<}
    For the sake of completeness, with a suitable adaptation of this procedure, Crasta in \cite{Cra2004} was able to prove that for any $\Omega \in \mathcal K_n$, it holds
\[
T(\Omega) < \int_{\Omega} d(x,\partial\Omega)^2\, dx.
\]
\end{oss}

\subsection{Other geometrical quantities}
In order to study the maximizing sequences of $\mathcal F$, we here introduce the following geometric quantities mentioned before in the Introduction

\begin{equation*}
\alpha(\Omega):=\frac{w_\Omega}{\diam(\Omega)}, \qquad \text{and}  \qquad \beta(\Omega) = \frac{P(\Omega)R_\Omega}{\abs{\Omega}}-1, \qquad \text{and} \qquad   \gamma(\Omega):=n-\frac{P(\Omega)R_\Omega}{\abs{\Omega}}.
\end{equation*}
We stress that $\beta(\Omega)$ and $\gamma(\Omega)$ are positive being for any $\Omega \in \mathcal K_n$ (see for instance \cite{AGS} and the references therein)
\[
1<\frac{P(\Omega)R_\Omega}{\abs{\Omega}}\le n
\]
Let $\{\Omega_\ell\}_{\ell\in \mathbb N}\subset\mathcal{K}_n$. We say that $\Omega_\ell$ is a sequence of thinning domains if
    \begin{equation}
\alpha(\Omega_\ell)\to0, \text{ for } \ell \to \infty.
    \end{equation}
\noindent In particular, we will call $\Omega_\ell$ a sequence of thinning cylinders if 
    \begin{equation}\label{thin_rect}
    \Omega_\ell = C \times \left[-\frac{1}{2\ell}, \frac{1}{2\ell}\right],
\end{equation}
    where $C\in \mathcal{K}_{n-1}$. Moreover, in the case $n=2$, the sequence \eqref{thin_rect} is called sequence of thinning rectangles.
We recall the following estimate, which is proved in \cite{fenchel_bonnesen} in the planar case and is generalized in \cite{DBN,Kova} to all dimensions.\\
As already mentioned in the Introduction, in \cite{AGS}, we proved that there exists a constant $C_0(n)=C_0$, such that
\begin{equation*}
    \beta(\Omega)\ge C_0\alpha(\Omega).
\end{equation*}
This means that every minimizing sequence of $\beta(\Omega)$ is a thinning domain, but not viceversa. Counterexamples can be given by sequences of flattening cones or pyramids. Considering $\gamma(\Omega)$, instead, it turns out that there is no strict relation between it and $\alpha(\Omega)$ or $\beta(\Omega)$. Indeed $\gamma(\Omega)= 0$ on tangential bodies and, therefore, not necessarily thinning, as for instance a ball.

\section{Proof of the Makai inequality}
\label{sec3}

This section is devoted to the proof of the Makai inequality in higher dimensions. We also emphasize the sharpness of this inequality.

\begin{proof}[Proof of Theorem \ref{thm:main}]
    In the sequel we will denote by $L(t)=P(t)^{\frac{1}{n-1}}$ and by $L\equiv L(0)=P(\Omega)^{\frac{1}{n-1}}$. Since $L(t)$ is concave, there exists a linear function $\lambda : t\in [0,R_\Omega]\to \mathbb R$, defined as
    \begin{equation*}
        \lambda(t) = L-at,\qquad \qquad a \in\mathbb R
    \end{equation*}
    such that
    \begin{equation}\label{eq:fixedmeasure}
        \abs{\Omega}= \int_0^{R_\Omega}L(t)^{n-1}=\int_0^{R_\Omega}\lambda(t)^{n-1}. 
    \end{equation}
    If we integrate \eqref{eq:fixedmeasure} by parts, i.e.
    \begin{equation*}
        \abs{\Omega}= \int_0^{R_\Omega}(L-at)^{n-1}\,dt = -\frac{(L-at)^n}{an}\bigg|_0^{R_\Omega},
    \end{equation*}
    we get a condition on $a$, that is
    \begin{equation}\label{eq:measureomegaa}
        \abs{\Omega}= \frac{L^n}{an}\bigg(1-\bigg(1-\frac{aR_\Omega}{L}\bigg)^n\bigg).
    \end{equation}
    If we consider the continuous function $f:\mathbb R \to \mathbb R$, defined as
\begin{equation*}
    f(a) = \frac{L^n}{an}\bigg(1-\bigg(1-\frac{aR_\Omega}{L}\bigg)^n\bigg)-\abs{\Omega}.
\end{equation*}
By the factorization valid for any $n\in\mathbb N$
\begin{equation*}
    b^n-c^n= (b-c)(b^{n-1}+b^{n-2}c+\cdots +bc^{n-2}+c^{n-1}),
\end{equation*}
we can rewrite the function as
\begin{equation*}
    f(a) = \frac{P(\Omega)R_\Omega}{n}\bigg[1+\bigg(1-\frac{aR_\Omega}{L}\bigg)+\cdots +\bigg(1-\frac{aR_\Omega}{L}\bigg)^{n-1}\bigg]-\abs{\Omega}.
\end{equation*}
From this form, it is possible to see that for $a<0$ we get
\begin{equation*}
    f(a) > P(\Omega)R_\Omega-\abs{\Omega} \ge 0,
\end{equation*}
while for $a> L/R_\Omega$
\begin{equation*}
    f(a) <\frac{P(\Omega)R_\Omega}{n}-\abs{\Omega}\le 0.
\end{equation*}
Moreover
\begin{equation*}
    \lim_{a\to 0^+}f(a)= P(\Omega)R_\Omega-\abs{\Omega}\ge 0,
\end{equation*}
and
\begin{equation*}
    f(L/R_\Omega) = \frac{P(\Omega)R_\Omega}{n}-\abs{\Omega}\le 0.
\end{equation*}
Therefore, by the zeros Theorem, there exists at least a $\Bar{a}\in [0, L/R_\Omega]$ such that $f(\Bar{a})=0$. This means that the only possible zeros of $f(\cdot)$ can be found in the interval $[0,L/R_\Omega]$, and only in this interval the the condition \eqref{eq:fixedmeasure} is well--posed. \\
Summarizing the properties of  $L(t)$ and $\lambda(t)$, we have that
\begin{enumerate}
\item $L(0)=\lambda(0)$;
    \item $L(t)$ is concave and $\lambda(t)$ is linear;
    \item $L(R_\Omega)\le \lambda (R_\Omega)$: in fact if by contradiction $L(R_\Omega)>\lambda(R_\Omega)$, then we would have $\|L\|_{L^{n-1}(0,R_\Omega)}> \|\lambda\|_{L^{n-1}(0,R_\Omega)}$, which is not possible since condition \eqref{eq:fixedmeasure} holds.
\end{enumerate} 
All these facts ensures that there exist a unique $\beta \in [0,R_\Omega]$ such that
\begin{align*}
    &L(t)-\lambda(t)\ge 0, \qquad\qquad t\in [0,\beta],\\
    &L(t)-\lambda(t)\le 0, \qquad\qquad t\in [\beta,R_\Omega].
\end{align*}
This implies that
\begin{equation*}
    \int_{0}^{R_\Omega}[L(t)^{n-1}-\lambda(t)^{n-1}]t^2\,dt \le \int_{0}^{\beta}[L(t)^{n-1}-\lambda(t)^{n-1}]\beta^2\,dt+\int_{\beta}^{R_\Omega}[L(t)^{n-1}-\lambda(t)^{n-1}]\beta^2\,dt= 0,
\end{equation*}
and therefore
\begin{equation}\label{eq:Llesslambda}
    \int_0^{R_\Omega}L(t)^{n-1}t^2\,dt\le \int_0^{R_\Omega}\lambda(t)^{n-1}t^2\,dt.
\end{equation}
Now, using proposition \ref{prop:TlessD} and applying Coarea formula, we have
\begin{equation*}
    T(\Omega)\le \int_{\Omega}d(x,\partial \Omega)^2\,dx = \int_0^{R_\Omega}P(t)t^2\,dt = \int_{0}^{R_\Omega}L(t)^{n-1}t^2\,dt.
\end{equation*}
Therefore, by \eqref{eq:Llesslambda}, we get
\begin{equation}\label{eq:Tlesslambda}
    T(\Omega) \le \int_{0}^{R_\Omega}\lambda(t)^{n-1}t^2\,dt.
\end{equation}
We now need to estimate integral \eqref{eq:Tlesslambda}. Let us stress that $\lambda(t) = L-at$, where $a \in [0,L/R_\Omega]$. Integrating by parts twice, we get
\begin{equation*}
  \begin{split}
  \int_{0}^{R_\Omega}&\lambda(t)^{n-1}t^2\,dt= \int_{0}^{R_\Omega}(L-at)^{n-1}t^2\,dt =\\
  &-\frac{(L-aR_\Omega)^nR_\Omega^2}{an}-2\frac{(L-aR_\Omega)^{n-1}R_\Omega}{a^2n(n+1)}+2\frac{L^{n+2}-(L-aR_\Omega)^{n+2}}{a^3n(n+1)(n+2)}=\\
  &\frac{2L^{n+2}}{a^3n(n+1)(n+2)}\bigg[-\frac{(n+1)(n+2)}{2}\bigg(1-\frac{aR_\Omega}{L}\bigg)^n \frac{a^2R_\Omega^2}{L^2}+\\
  &\qquad\qquad \qquad-(n+2)\bigg(1-\frac{aR_\Omega}{L}\bigg)^{n+1} \frac{aR_\Omega}{L}+ \bigg(1-\bigg(1-\frac{aR_\Omega}{L}\bigg)^{n+2} \bigg)\bigg].
  \end{split} 
\end{equation*}
From now on, we will change variable, writing
\begin{equation*}
    y= \frac{aR_\Omega}{L},     \qquad y \in [0,1].
\end{equation*}
Therefore
\begin{equation*}
    \int_{0}^{R_\Omega}\lambda(t)^{n-1}t^2\,dt= \frac{2L^{n+2}}{a^3n(n+1)(n+2)}\bigg[-\frac{(n+1)(n+2)}{2}(1-y)^ny^2-(n+2)(1-y)^{n+1}y+(1-(1-y)^{n+2})\bigg].
\end{equation*}
We want to subtract to this quantity, the conjectured optimal constant, that is 
\begin{equation*}
    \frac{2n^2}{(n+1)(n+2)}\frac{\abs{\Omega}^3}{P(\Omega)^2}.
\end{equation*}
Now, by \eqref{eq:measureomegaa}, we have that
\begin{equation*}
    \abs{\Omega}= \frac{L^n}{an}(1-(1-y)^n),
\end{equation*}
therefore
\begin{equation*}
     \frac{2n^2}{(n+1)(n+2)}\frac{\abs{\Omega}^3}{P(\Omega)^2}= \frac{2L^{n+2}}{a^3n(n+1)(n+2)}(1-(1-y)^n)^3.
\end{equation*}
Eventually 
\begin{equation}
\label{pointquant}
\begin{split}
    &T(\Omega)-\frac{2n^2}{(n+1)(n+2)}\frac{\abs{\Omega}^3}{P(\Omega)^2}\le\\
    &\int_{\Omega} d(x, \partial \Omega)^2 \, dx-\frac{2n^2}{(n+1)(n+2)}\frac{\abs{\Omega}^3}{P(\Omega)^2}\le\\
    &\frac{2L^{n+2}}{a^3n(n+1)(n+2)}\bigg[-\frac{(n+1)(n+2)}{2}(1-y)^ny^2+\\
    &-(n+2)(1-y)^{n+1}y
    +(1-(1-y)^{n+2})-(1-(1-y)^n)^3\bigg].    
\end{split}
\end{equation}
If we prove that the quantity in the squared brackets is non-positive, we have finished. Let us denote by $z=1-y$, $z\in[0,1]$ and let us denote by 
\begin{equation*}
    g(z) = -\frac{(n+1)(n+2)}{2}z^n(1-z)^2-(n+2)z^{n+1}(1-z)
    +(1-z^{n+2})-(1-z^n)^3
\end{equation*}
the quantity inside the squared brackets. After straightforward computation, we get
\begin{equation*}
    g(z) = \frac{z^n}{2}[2z^{2n}-6z^n-n(n+1)z^2+2n(n+2)z-(n^2+3n-4)].
\end{equation*}
Denoting by
\begin{equation*}
    h(z) = 2z^{2n}-6z^n-n(n+1)z^2+2n(n+2)z-(n^2+3n-4),
\end{equation*}
the term in the square brackets, we notice that $h(0)<0$ and $h(1)=0$. Moreover computing the first and second derivative of $h$, we get
\begin{align*}
    &h'(z) = 4nz^{2n-1}-6nz^{n-1}-2n(n+1)z+2n(n+2),\\
    &h''(z) = 4n(2n-1)z^{2n-2}-6n(n-1)z^{n-2}-2n(n+1). 
\end{align*}
We stress that $h'(0)>0$, $h''(0)<0$ and $h'(1)=h''(1)=0$.  Computing the third derivative, we have
\begin{equation*}
h'''(z) = z^{n-3}2n(n-1)\bigg[4(2n-1)z^n-3(n-2)\bigg],
\end{equation*}
and
\begin{equation*}
    h'''(z)\le 0 \qquad \iff\qquad 0\le z\le \Tilde{z}= \bigg(\frac{3(n-2)}{4(2n-1)}\bigg)^{\frac1n}\in [0,1).
\end{equation*}
This implies that $\sup_{[0,1]}h''(z)= \sup\{h''(0),h''(1)\}=h''(1)=0$. As a consequence $h''(z)\le 0$ and therefore $h'(z)$ is decreasing. In turns this implies that $\inf_{[0,1]}h'(z)=h'(1)=0$, which gives $h'(z)>0 $ for all $z\in [0,1]$. This implies that $\sup_{[0,1]}h(z)= h(1)=0$ and therefore $h(z)\le 0$.

As noted in the Introduction, the sharpness of the inequality has been already proved in \cite{buttazzo2020convex}, where they prove that sequences of cones reach the desired constant. To make the paper self contained we prove it making a direct computation. 
Let $\Omega_k$ be a sequence of collapsying cones, i.e.
    \begin{equation*}
        \Omega_k = \{(x,y)\in \mathbb R^{n-1}\times\mathbb R: x\in B_1^{n-1}, \; 0\le y \le k^{-1}(1-\abs{x})\},
    \end{equation*}
    where $B_1^{n-1}$ is the $(n-1)$-dimensional  unit ball centered at the origin. For $k$ large enough, $\{\Omega_k\}_{k\in \mathbb N}$ is a sequence of thinning domains, so that the torsion acquires a nice form approximated at the first order (see \cite{BFreitas, buttazzo2020convex}), that is
    \begin{equation*}
        T(\Omega_k) \sim \frac{1}{12k^3}\int_{B_1^{n-1}}(1-\abs{x})^3\,dx.
    \end{equation*}
    Applying Coarea Formula and integrating by parts  three times we get
    \begin{equation*}
     \int_{B_1^{n-1}}(1-\abs{x})^3\,dx= (n-1)\omega_{n-1}\int_0^1(1-t)^3t^{n-2}\,dt = \frac{6\omega_{n-1}}{n(n+1)(n+2)}.  
    \end{equation*}
    Therefore
    \begin{equation*}
        T(\Omega_k) \sim \frac{\omega_{n-1}}{2n(n+1)(n+2)k^3}.
    \end{equation*}
    Moreover we have that
    \begin{equation*}
        P(\Omega_k) \sim 2|B_1^{n-1}| = 2\omega_{n-1}
    \end{equation*}
    and
    \begin{equation*}
        |\Omega_k|= \frac{1}{k}\int_{B_1^{n-1}}(1-\abs{x})\,dx = \frac{\omega_{n-1}}{nk}.
    \end{equation*}
    This means that for $k\to +\infty$, we have
    \begin{equation*}
        \frac{T(\Omega_k)P(\Omega_k)^2}{\abs{\Omega_k}^3}\sim \frac{\omega_{n-1}}{2n(n+1)(n+2)k^3}\frac{4(\omega_{n-1})^2}{(\omega_{n-1})^3}n^3k^3 = \frac{2n^2}{(n+1)(n+2)},
    \end{equation*}
    proving the sharpness of the inequality. 
\end{proof}
    
    \section{On the optimizing sequences}

    \label{sec4}

In this section, we turn our attention to the stability question, which can be formulated as follows:
is it true that
\begin{center}
    \emph{
    if $\mathcal{F}(\Omega)$ is close to $2n^2/((n+1)(n+2))$
    for a convex set $\Omega$, then in some sense $\Omega$ is
    a thin tangential body?
    }
\end{center}
This kind of question is not new: the same problem was posed by Crasta in \cite{Cra2004}, although no formal proof was provided there. Similar questions were also considered for the lower bound of the same functional in \cite{AGS,AMPS}, and for the first nontrivial Neumann eigenvalue in \cite{ABF,ABF2}.\\
\noindent Below we exploit the proof of the Makai inequality given in the previous sections, looking for room for improvement. In particular, the inequality is obtained via the following chain of inequalities:
\begin{equation*}
\mathcal{F}(\Omega) \leq \frac{ P(\Omega)^2}{\abs{\Omega}^3}\int_{\Omega}d(x,\partial\Omega)^2\, dx  
\leq \frac{2n^2}{(n+1)(n+2)}.
\end{equation*}
It turns out that our question can be split into two different ones: what is a necessary condition for which 
\(
\mathcal{F}(\Omega)\)
is close to 
\(
\frac{ P(\Omega)^2}{\abs{\Omega}^3}\int_{\Omega}d(x,\partial\Omega)^2\, dx,
\)
and moreover, what is a necessary condition for which 
\(
\frac{ P(\Omega)^2}{\abs{\Omega}^3}\int_{\Omega}d(x,\partial\Omega)^2\, dx
\)
is close to 
\(
\frac{2n^2}{(n+1)(n+2)}?
\)

As for the first question, we provide a partial affirmative answer, which is not quantitative but only asymptotic: it states, roughly speaking, that whenever 
\(\mathcal{F}(\Omega)\)
approaches 
\[
\frac{ P(\Omega)^2}{\abs{\Omega}^3}\int_{\Omega}d(x,\partial\Omega)^2\, dx,
\]
the convex set $\Omega$ tends to become thin. It can be stated as in Theorem \ref{thm:supremum}.

\begin{proof}[Proof
    of Theorem \ref{thm:supremum}]Let us assume without lost generality that $\diam (\Omega_k)=1$ and 
    $$ \lim_k \left(\int_{\Omega_k}d(x,\partial \Omega_k)^2 \,dx- T(\Omega_k)\right)\frac{P^2(\Omega_k)}{\abs{\Omega_k}^3} \leq \lim_k \left(\frac{2n^2}{(n+1)(n+2)} - T(\Omega_k)\frac{P^2(\Omega_k)}{\abs{\Omega_k}^3}\right)=0 ,$$
        and let us suppose by contraddiction that $$\limsup_k \frac{w(\Omega_k)}{\diam(\Omega_k)} = c >0.$$ 
        Then, up to a subsequence, that we will always indicate by $\{\Omega_k\}_k$, we have
    $$\lim_k \frac{w(\Omega_m)}{\diam(\Omega_k)} = c >0.$$
    Therefore $\frac{w(\Omega_k)}{\diam(\Omega_k)} > c/2$ for all $k \in \mathbb N$. Since $R(\Omega_k) \geq c_n w(\Omega_k) \geq c_n \frac{c}{2}$, then every
     $\Omega_k \supset B_{c_n \frac{c}{2}}$. Moreover, we know that
     \begin{equation*}
         \abs{\Omega} \ge \frac{R_{\Omega}}{n}P(\Omega),
     \end{equation*}
     
     and that \begin{equation*}
         \abs{\Omega}\le |B_{\diam(\Omega)/2}|= \frac{\omega_n}{2^n}\diam(\Omega)^n.
     \end{equation*}
     Then, the diameter constraint and the previous two inequalities imply
     \begin{equation*}
        P(\Omega_k) \le \frac{n\omega_n}{2^nR_{\Omega_k}}\le \frac{n\omega_n}{2^{n-1}c_nc}<+\infty \qquad \forall k\in \mathbb N.
     \end{equation*}
    Therefore the sequence $\{\Omega_k\}_k$ is equibounded and we can apply the Blaschke selection theorem (see    \cite[theorem 1.8.6]{schneider}). Hence, there exists a subsequence of $\Omega_k$ converging to a set $\tilde \Omega$ in the Hausdorff sense.
By the continuity of the Torsion, the perimeter and the volume with respect to the Hausdorff distance, we have  
    $$ \lim_k \left(\int_{\Omega_k}d(x,\partial \Omega_k)^2 \, dx - T(\Omega_k)\right)\frac{P^2(\Omega_k)}{\abs{\Omega_k}^3}= 0= \left(\int_{\tilde\Omega}d(x,\partial \tilde \Omega)^2 \, dx - T(\tilde \Omega)\right)\frac{P^2(\tilde \Omega)}{\abs{\tilde \Omega}^3}.$$
This is a contradiction with Remark \ref{crasta<}.

    \end{proof}

As for the second question, we show that 
\[
\frac{ P(\Omega)^2}{\abs{\Omega}^3}\int_{\Omega}d(x,\partial\Omega)^2\, dx
\]
is close to $\frac{2n^2}{(n+1)(n+2)}$ if and only if $\Omega$ is close to a tangential body, that is, the quantity
\[
\gamma(\Omega) =  n-\frac{P(\Omega)R_\Omega}{\abs{\Omega}}
\]
is small. These informations can be stated in the followig proposition.

\begin{prop}\label{prop:quantitative}
    Let $n\ge 2$ and let $\Omega \in \mathcal{K}_n$. Then,
    \begin{equation*}
        C_1\gamma(\Omega)^n\le \frac{2n^2}{(n+1)(n+2)}-\frac{P(\Omega)^2}{\abs{\Omega}^3} \int_{\Omega}d(x,\Omega)^2\, dx\le C_2\gamma(\Omega),
    \end{equation*}
      where
    \begin{equation*}
        C_1(n) = \frac{n^2+3n-4}{n(n+1)(n+2)(n-1)^{n}} \qquad \text{ and } \qquad C_2(n) = \frac{6n}{(n+1)(n+2)}. 
    \end{equation*}
\end{prop}

\begin{proof}
   We start from \eqref{pointquant} in the qualitative proof, which can be rewritten as
   
\begin{equation}\label{eq:firststepquantitative}
\int_{\Omega} d(x,\partial\Omega)^2\, dx-\frac{2n^2}{(n+1)(n+2)}\frac{\abs{\Omega}^3}{P(\Omega)^2}\le
    \frac{2L^{n+2}}{a^3n(n+1)(n+2)}g(z),
\end{equation}
where
\begin{equation*}
    g(z) = \frac{z^n}{2}h(z),
\end{equation*}
and $h(z)$ defined above which is negative. In particular we have seen that $z=1$ is a zero of order $3$ for $h$, therefore we can write
\begin{equation*}
    h(z) = -(1-z)^3H(z),
\end{equation*}
for some function $H(z)$. In particular 
\begin{equation*}
    H(z) = -\frac{h(z)}{(1-z)^3}
\end{equation*}
and it is possible to see that
\begin{equation*}
    H'(z) = \frac{-h'(z)(1-z)-3h(z)}{(1-z)^4}>0.
\end{equation*}
This means that 
\begin{equation*}
    \inf_{[0,1]}H(z) = H(0) = -h(0)= n^2+3n-4.
\end{equation*}
Therefore
\begin{equation*}
    h(z) \le -(n^2+3n-4)(1-z)^3.
\end{equation*}
and consequently
\begin{equation}\label{eq:estimateg}
    g(z) \le -\frac{n^2+3n-4}{2}(1-z)^3z^n.
\end{equation}
We want to multiply \eqref{eq:firststepquantitative} by $P^2(\Omega)/\abs{\Omega}^3$, where we recall that $P(\Omega) =L^{n-1}$ and $\abs{\Omega}$ is given by 
\begin{equation}\label{eq:measureomegawithz}
    \abs{\Omega}= \frac{L^n}{an}(1-z^n).
\end{equation}
This implies that
\begin{equation}\label{eq:pmz}
    \frac{P(\Omega)^2}{\abs{\Omega}^3}=\frac{a^3n^3}{L^{n+2}(1-z^n)^3}.
\end{equation}
Therefore, after multiplying by the desired quantity and plugging \eqref{eq:pmz} and \eqref{eq:estimateg} into \eqref{eq:firststepquantitative}, we get
\begin{equation*}
    \int_{\Omega} d(x,\partial\Omega)^2\, dx\frac{P(\Omega)^2}{\abs{\Omega}^3}-\frac{2n^2}{(n+1)(n+2)}\le -\frac{n^2(n^2+3n-4)}{(n+1)(n+2)}\cdot\frac{(1-z)^3z^n}{(1-z^n)^3}.
\end{equation*}
Let us stress that since $0\le z\le 1$
\begin{equation*}
    (1-z^n)^3= (1-z)^3(1+z+z^2+\cdots+z^{n-1})^3\le n^3(1-z)^3.
\end{equation*}
Hence
\begin{equation}\label{eq:secondstepquantitative}
   \int_{\Omega} d(x,\partial\Omega)^2\, dx \frac{P(\Omega)^2}{\abs{\Omega}^3}-\frac{2n^2}{(n+1)(n+2)}\le -\frac{n^2+3n-4}{n(n+1)(n+2)} z^n.
\end{equation}
We just need to estimate $z^n$. 

Let us denote by $\gamma(\Omega) \in [0,n-1]$ the remainder term
\begin{equation*}
    \tilde \gamma(\Omega)= \frac{n\abs{\Omega}}{P(\Omega)R_\Omega}-1.
\end{equation*}
Recalling that $1-z= y= aR_\Omega/L$, then
\begin{equation}\label{eq:ainz}
    a = (1-z)\frac{L}{R_\Omega}.
\end{equation}
Recalling \eqref{eq:measureomegawithz} and using \eqref{eq:ainz}, we get
\begin{equation*}
    \tilde \gamma(\Omega)+1 = \frac{1-z^n}{1-z}.
\end{equation*}
But again
\begin{equation*}
   \frac{1-z^n}{1-z}= \frac{(1-z)(1+z+\cdots +z^{n-1})}{1-z}= 1+z+\cdots +z^{n-1},
\end{equation*}
which gives
\begin{equation*}
    \tilde \gamma(\Omega) = z+z^2+\cdots+z^{n-1}.
\end{equation*}
Since $0\le z \le 1$, we have that $z^k\le z$ for all $k\ge 1$ and so
\begin{equation*}
    \tilde \gamma(\Omega)\le (n-1)z. 
\end{equation*}
In particular
\begin{equation}\label{eq:estimatezn}
    z^n \ge \frac{ \tilde \gamma(\Omega)^n}{(n-1)^n}.
\end{equation}
Using \eqref{eq:estimatezn} in \eqref{eq:secondstepquantitative}, we have
\begin{equation*}
\begin{aligned}
    \int_{\Omega} d(x,\partial\Omega)^2\, dx\frac{P(\Omega)^2}{\abs{\Omega}^3}-\frac{2n^2}{(n+1)(n+2)}&\le -\frac{n^2+3n-4}{n(n+1)(n+2)(n-1)^{n}}\tilde \gamma(\Omega)^n \\&\leq -\frac{n^2+3n-4}{n(n+1)(n+2)(n-1)^{n}} \gamma(\Omega)^n ,
\end{aligned}
\end{equation*}
where
$$\gamma(\Omega)= n - \frac{P(\Omega) R_{\Omega}}{\abs{\Omega}}.$$
Consequently, we have \begin{equation*}
   \frac{2n^2}{(n+1)(n+2)}-\int_{\Omega} d(x,\partial\Omega)^2\, dx\frac{P(\Omega)^2}{\abs{\Omega}^3}\ge \frac{n^2+3n-4}{n(n+1)(n+2)(n-1)^{n}}\gamma(\Omega)^n, 
\end{equation*}
which is the desired result.

On the other hand, by \eqref{eq:concavity}
$$\int_{\Omega} d(x,\partial\Omega)^2\, dx = \int_{0}^{R_\Omega} t^2 P(t)\, dt \geq \int_{0}^{R_\Omega} t^2 \left(1- \frac{t}{R_\Omega}\right)^{n-1}P(\Omega)\, dt = \frac{2R_{\Omega}^3(P(\Omega))}{n(n+1)(n+2)}.$$
Hence, we have
\begin{equation*}
\begin{aligned}
    \frac{P(\Omega)^2}{\abs{\Omega}^3}\int_{\Omega} d(x,\partial\Omega)^2\, dx &\geq \frac{2}{n(n+1)(n+2)} \left(n^3+\frac{R^3_{\Omega}P(\Omega)^3}{\abs{\Omega}^3}-n^3\right)\\
    &= \frac{2}{n(n+1)(n+2)} - \frac{2n^2}{(n+1)(n+2)}\left( n^3 -\frac{R^3_{\Omega}P(\Omega)^3}{\abs{\Omega}^3}\right)\\
    &\geq \frac{2n^2}{(n+1)(n+2)} - \frac{6n}{(n+1)(n+2)}\left( n-\frac{R_{\Omega}P(\Omega)}{\abs{\Omega}}\right).
    \end{aligned}
\end{equation*}
\end{proof}

 Theorem \ref{thm:quantitative} is a direct consequence of Proposition \ref{prop:quantitative}, since $T(\Omega) \leq \int_{\Omega} d(x,\partial\Omega)^2\, dx$.

 \begin{proof}[Proof of Theorem \ref{thm:quantitative}]
 Since $T(\Omega) \leq \int_{\Omega} d(x,\partial\Omega)^2\, dx$,
$$ \frac{2n^2}{(n+1)(n+2)}-\mathcal{F}(\Omega) \geq \frac{2n^2}{(n+1)(n+2)}-\frac{P(\Omega)^2}{\abs{\Omega}^3} \int_{\Omega} d(x, \partial \Omega)^2 \, dx \geq C_1\gamma(\Omega)^n$$
 \end{proof}
\noindent Eventually, the proof of Corollary \ref{Corol}, is a direct consequence of Theorems \ref{thm:quantitative} and \ref{thm:supremum}.\\
As last result of this section we show that there is a particular class of flattening sets for which $$\lim_k \left(\int_{\Omega_k}d(x,\partial \Omega_k)^2\,dx - T(\Omega_k)\right)\frac{P^2(\Omega_k)}{\abs{\Omega_k}^3}=0.$$
Indeed we get the following result holds true.

\begin{prop}Let $n\ge 2$ and let $\Omega \in \mathcal{K}_n$. Then,
    $$\left(\int_{\Omega}d(x,\partial \Omega)^2\,dx - T(\Omega)\right)\frac{P^2(\Omega)}{\abs{\Omega}^3}  \leq \frac{n^2+n+1}{3} \beta(\Omega).$$
\end{prop}
\begin{proof} Using \eqref{polyatorsion} and the expression of the integral of $d(x,\Omega)^2$, we get
$$\begin{aligned}
    \left(\int_{\Omega}d(x,\partial \Omega)^2\,dx - T(\Omega)\right)\frac{P^2(\Omega)}{\abs{\Omega}^3}  &\leq \int_0^{R_\Omega} t^2 P(t)\, dt \frac{P^2(\Omega)}{\abs{\Omega}^3} -\frac1 3 \\&\leq \frac{R_\Omega^3 P(\Omega)^3}{3\abs{\Omega}^3}  - \frac 1 3\leq  \frac{n^2+n+1}{3} \beta(\Omega). 
\end{aligned}$$
\end{proof}

\section*{Acknowledgments}

The authors were partially supported by Gruppo Nazionale per l’Analisi Matematica, la Probabilità e le loro Applicazioni
(GNAMPA) of Istituto Nazionale di Alta Matematica (INdAM).\\
Rossano Sannipoli was supported by the grant no. 26-21940S
of the Czech Science Foundation.
\bibliographystyle{plain}
\bibliography{biblio}

\Addresses
 
\end{document}